\documentclass[12pt, a4paper]{article}
\usepackage[utf8x]{inputenc}
\usepackage{lscape,fancyhdr}
\usepackage{amsfonts}
\usepackage{amssymb,amsthm,amsmath,color,mathrsfs}
\usepackage{lmodern}
\usepackage{enumerate}
\usepackage{graphicx}
\usepackage{multirow}
\usepackage{hyperref}
\hypersetup{colorlinks=true,linkcolor=blue!85!black}
\usepackage{tikz}
\usepackage{tikz-cd}

\newtheorem{Th}{Theorem}[section]
\newtheorem{Cor}[Th]{Corollary}
\newtheorem{Lem}[Th]{Lemma}
\newtheorem{Prop}[Th]{Proposition}

\theoremstyle{definition}
\newtheorem{Def}[Th]{Definition}

\newtheorem{Rem}[Th]{Remark}

\usepackage[left=4cm,right=1cm,top=4cm,bottom=1cm]{geometry}

\newcommand{\Title}[1]{\begin{center}\Large \textbf{#1}\end{center}}

\newcommand{\Author}[3]{\begin{center}\large #1\\ {\small \textit{#2}, #3}\end{center}}

\newcommand\blfootnote[1]{%
	\begingroup
	\renewcommand\thefootnote{}\footnote{#1}%
	\addtocounter{footnote}{-1}%
	\endgroup}

\newcommand{\AMS}[1]{AMS Subject Classification : #1}

\newcommand{\Kex}[1]{Kex words : #1}

\newcommand{\Ack}[1]{\noindent\textbf{\underline{Acknowledgement} :} #1}

\begin{document}

\Title{On Tempered Ultradistributions in Classical Morrey Spaces}
	
	\Author{Anslem Uche Amaonyeiro}{Department of Mathematics, University, Ibadan, Nigeria}{anslemamaonyeiro@uam.edu.ng}
\Author{Murphy Egwe*}{Department of Mathematics, University of Ibadan, Nigeria}{murphy.egwe@ui.edu.ng}*\blfootnote{*Corresponding author.}	
		
% xou can add more author(s) bx adding command \Author{arg1}{arg2}{arg3}	

\noindent\rule{\textwidth}{1pt}

%\blfootnote{\acceptance{****}{****}}\vspace{-.75 cm}

\begin{abstract}
\noindent We introduce the notion of tempered ultradistributions in classical Morrey spaces by preserving their respective properties. Moreover we investigate some embedding results within the scale of classical Morrey spaces (local Morrey space $\mathcal{L}^{p,\beta}(\mathbb{C}^{n},\mu)$ or global Morrey space $L^{p,\beta}(\mathbb{C}^{n})$ where the underlying functions $\mu$ satisfy the growth condition by polxnomials. Finally we present some results on the new classical Morrey spaces described by tempered ultradistributions satisfying both the vanishing and slow growth conditions in form of embedding theorems.

\end{abstract}

\AMS{46F05. 46F10. 46E10. 42B35. 46E30}

\Kex{Morrey space. Local Morrey space. Global Morrey space. Tempered ultradistributions. Embeddings}

\section{Introduction}
\noindent The notion of well-known Morrey spaces was introduced bx C. Morrey in 1938 \cite{18} as a result of the study of partial differential equations involving the local behaviour of the solutions to nonlinear elliptic equations. Several works have been performed involving the structural properties of function theory in these spaces, refer to \cite{1*} for further discussion. Functions described in these spaces enjoy nice properties than those in $L^p$-spaces. Such properties could include the boundedness property as established by some functions like $\vert\xi\vert^{\alpha}$ for $\alpha=-n/p$ even though Morrey spaces are non-separable. Further studies on Morrey spaces were later developed by Campanato \cite{9}, Peetre \cite{20} and Brudnyi \cite{8}. In \cite{2}, C. Morrey considered the theory of the integral average over a ball with a certain growth condition.\\

\noindent In \cite{29} some set of functions in Morrey space with continuous translation in its Morrey norm play a significant impact in approximating or describing functions in the Morrey spaces. This concept was also discussed in \cite{11} and \cite{15}. Some inclusions of Morrey spaces were carried out.\\

\noindent The concept of generalized functions caused a significant approach in the research \cite{27}, and this theory has been developed over the years. The notion of ultradistributions as a new class of generalized functions are useful in many areas including different areas of analysis, quantum field, time-frequency analysis and many others. The extended form of generalized functions called ultradistributions has become one of the top analysis areas of mathematics. These spaces were developed to solve smooth and non-smooth problems in the theory of partial differential equations as they also exhibit additional property of Taylor series expansion. The theory of tempered ultradistributions plays a significant role in analysis. For details on tempered ultradistributions and their properties, see details in \cite{3} and \cite{5}. \\

\noindent This work carries out a survey on ultradistributions with polynomial growth called tempered ultradistributions in classical Morrey spaces with emphasis on the parameters. Some new Morrey spaces with slow growth conditions are introduced. Some known subspaces or pre-duals of Morrey spaces with vanishing properties are discussed. These properties combined with the slow growth conditions allow us to show all elements in the Morrey spaces may be approximated by tempered ultradistributions in the Morrey norm which is more general than tempered ultradistributions. For details on some of the subspaces of Morrey spaces, we refer to \cite{4}. \\

\noindent We organize this paper as follows: section 2 gives some insights on some notations and preliminaries concerning Morrey spaces and tempered ultradistribution spaces. Section 3 dwells on the new Morrey spaces and their properties including the translation invariancy with respect to convolution, and continuous embeddings. Section 4 discusses some main results surrounding the continuous inclusions or embedding.

\section{Preliminaries}
We begin with some definitions which shall be needed in the sequel.
\subsection{Classical Morrey Space}
\begin{Def}
\label{Def1}
Let $p\in \mathbb{Z}^{+}$, $1\leq p<\infty$ and $0\leq \beta\leq n$, $n\in\mathbb{N}$. We define the global classical Morrey space (also known as the homogeneous Morrey space), denoted by $L^{p,\beta}$, as
\begin{equation}
\label{eqn2.1**}
L^{p,\beta}(\Omega):=\Big\{ f\in L_{\text{loc}}^{p}:\sup_{y\in\Omega,r>0}\Big(\frac{1}{r^\beta}\int_{\bar{B}(y,r)}\vert f(x)\vert^{p}dx\Big)^{\frac{1}{p}}<\infty\Big\}
\end{equation}
where $L_{\text{loc}}^{p}$ denotes the space of all locally $p$-th integrable functions, $\bar{B}(x,r)=B(x,r)\cap\Omega$ for an open subset $\Omega$ of $n$-dimensional real or complex plane.
\end{Def}
 In a similar fashion we define the local version of the Morrey space.
\begin{Def}
\label{Def2}
Let $p\in [1,\infty)$ and $\beta\in[0,n)$. We define the local classical Morrey space (the inhomogeneous type), denoted $\mathcal{L}^{p,\beta}$ of $L^{p,\beta}$, as
\begin{equation}
\label{eqn2.1***}
\mathcal{L}^{p,\beta}(\Omega):=\Big\{ f\in L_{\text{loc}}^{p}:\sup_{y\in\Omega,r>0}\Big(\frac{1}{r^\beta}\int_{\bar{B}(y,r)}\vert f(x)\vert^{p}dx\Big)^{\frac{1}{p}}<\infty\Big\}
\end{equation}
\end{Def}
These spaces defined in (\ref{eqn2.1**}) and (\ref{eqn2.1***}) are also called complete normed spaces (also known as Banach spaces) equipped with the respective norms

\begin{equation}
\label{eqn2.3}
\Vert f\Vert_{L^{p,\beta}(\Omega)}:=\sup_{y\in\Omega;r>0}\Big(r^{-\beta}\int_{\bar{B}(y,r)}\vert f(x)\vert^{p}dx\Big)^{\frac{1}{p}}
\end{equation}

\begin{equation}
\label{eqn2.3*}
\Vert f\Vert_{\mathcal{L}^{p,\beta}(\Omega)}:=\sup_{y\in\Omega;r>0}\Big(\frac{1}{r^\beta}\int_{\bar{B}(y,r)}\vert f(x)\vert^{p}dx\Big)^{\frac{1}{p}}
\end{equation}
We write $\Vert\cdot\Vert_{p,\beta}$ for short to denote the Morrey norms defined in (\ref{eqn2.3}) and (\ref{eqn2.3*}) when it is clear which of the versions will be considered in each situation. In the subsequent definitions, we will use the complex plane $\mathbb{C}^n$. We can as well make reference to the parameter $\beta$.

\begin{Rem}
\label{rem*}
\begin{itemize}
\item[(i)] The parameters $p$ and $\beta$ play major role in the notion of classical Morrey space. We observed that $``p''$ plays the role of local integrability since $L^{p,\beta}$ is a subset of $L^{p}_{\text{loc}}$, and parameter $\beta$ seems to be the dilation index.
\item[(ii)] S. Campanato \cite{9} defined the Morrey space $L^{p,\beta}$ such that the following norm is finite:
\begin{equation}
\label{eqn5}
\Vert f\Vert_{L^{p,\beta}}=\Big(\sup_{(x,y)\in\mathbb{R}^{n}\times\mathbb{R}}r^{\beta-n}\int_{B(x,r)}\vert f(y)-f_{r}(x)\vert^{p}dy\Big)^{\frac{1}{p}}
\end{equation}
He observed from (\ref{eqn5}) that if $\beta=n$ then $L^{p,\beta}$ coincides with $L^p$; if $0<\beta<n$ then the local version $\mathcal{L}^{p,\beta}$ equals the global version $L^{p,\beta}$; if $\beta=0$, then $\mathcal{L}^{p,\beta}$ represents the space of all functions with bounded mean oscillation (BMO); and if $-p<\beta<0$, then $\mathcal{L}^{p,\beta}$ represents the space $C^{\alpha}$-the space of all functions satisfying the H$\ddot{\text{o}}$lder inequality with exponent $\alpha=-\beta/p$.
\item[(iii)]From definitions (\ref{Def1}) and (\ref{Def2}), Morrey spaces can be viewed as generalization of the $L^{p}$-spaces, and they constitute the subspaces of Morrey-Campanato spaces, H$\ddot{\text{o}}$lder spaces and the space BMO.
\item[(iv)] Few observations have been drawn from the classical Morrey spaces which can be highlighted as below:\\
Given the two versions $L^{p,\beta}$ and $\mathcal{L}^{p,\beta}$. If $\beta$ is strictly larger than the dimension $n$ of the underlying Euclidean space, that is, $\beta>n$, then
\[\mathcal{L}^{p,\beta}=L^{p,\beta}=\lbrace 0\rbrace\]
Again we have some following observations. For $\beta=n$, then $L^{p,n}(\Omega)=L^{\infty}(\Omega)=\mathcal{L}^{p,n}(\Omega)$ and for $\beta=0$, it can be written as $\mathcal{L}^{p,0}(\Omega)=\mathcal{L}^p$ and $L^{p,0}=L^{p}$.\\
Some continuous embedding can be derived as
\[L^{p,\beta}(\Omega)\hookrightarrow \mathcal{L}^{\beta}(\Omega)\quad\text{and}\quad L^{\infty}(\Omega)\hookrightarrow \mathcal{L}^{p,\beta}(\Omega).\]
\end{itemize}
\end{Rem}

Consider the following result.

\begin{Th}
\label{thmp}
Let $1\leq p\leq p'<\infty$ and $1\leq \beta,u\leq n$ be given. Then we have the continuous embedding
\begin{equation}
\label{eqnp}
L^{p',u}(\Omega)\hookrightarrow L^{p,\beta}
\end{equation}
\end{Th}

\begin{proof}
An application of the H$\ddot{\text{o}}$lder inequality (see Theorem \ref{th5.6}) yields (\ref{eqnp}) under the condition $(\beta-n)p'=p(u-n)$.
\end{proof}

\begin{Cor}
\label{corq}
Let $p,p'<\infty$, $p,p'\geq 1$, $p\leq p'$ and $1\leq \beta,u\leq n$. Assume that $\vert\Omega\vert=\infty$. Then we have
\[\mathcal{L}^{p',u}(\Omega)\hookrightarrow \mathcal{L}^{p,\beta}(\Omega)\]
with the following condition $\displaystyle (\beta-n)p'\leq (u-n)p$.
\end{Cor}

\begin{Cor}
\label{corr}
From Theorem \ref{thmp}, assume that the domain $\Omega$ is bounded such that $\vert\Omega\vert<\infty$. Then we obtain the embedding
\[L^{\infty}(\Omega)\hookrightarrow L^{p,\beta}(\Omega)\hookrightarrow L^{p}(\Omega).\]
\end{Cor}

\begin{Rem}
Some calculations involving Morrey norm establishes the following property from the homogeneous Morrey space:
\begin{equation}
\label{eqnref}
\Vert g(s\cdot)\Vert_{L^{p,\beta}(\Omega)}\equiv s^{\frac{\beta-n}{p}}\Vert g\Vert_{L^{p,\beta}(\Omega)},\quad s>0.
\end{equation}
Also the application of the Young's convolution inequality is true for homogeneous Morrey space $L^{p,\beta}(\mathbb{R}^{n})$:
\begin{equation}
\label{eqnref*}
\Vert f*g\Vert_{p,\beta}\leq \Vert g\Vert_{1}\Vert f\Vert_{p,\beta},\quad 0\leq \beta\leq n,\quad 1\leq p,\infty.
\end{equation}
Again with condition $1<p<\infty$, $0<\beta<n$, we have the following chain of continuous embedding
\begin{equation}
\label{eqnS}
\mathcal{S}\hookrightarrow L^{\frac{pn}{n-\beta}}\hookrightarrow \mathcal{L}^{p,\beta}\hookrightarrow\mathcal{S}'
\end{equation}
where $\mathcal{S}$ denotes the space of all real-valued rapidly decreasing functions and $\mathcal{S}'$ denotes the space of all tempered distributions.\\
it is observed that if $\beta>0$ then the Morrey spaces are non-separable. Also, the usual classes of smooth functions like $C^{\infty}_{0}$ or $\mathcal{S}$ has its closure in $L^{p,\beta}$ and $\mathcal{L}^{p,\beta}$ but the space $C^{\infty}_{0}$ is dense in the dual of $L^{p,\beta}$, denoted $(L^{p,\beta})^{*}$, meaning that $(L^{p,\beta})^{*}$ denotes the closure of all $C_{0}^{\infty}$-functions in Morrey norm. Recall that
\begin{equation}
\label{eqn7}
\Big( (L^{p,\beta})^{*} \Big)''=\Big( (L^{p,\beta})^{*'} \Big)'=L^{p,\beta}\quad\text{for}\quad 1<p<\infty, 0\leq \beta<n.
\end{equation}
\end{Rem}

In what follows we introduce some of the duals and preduals of the classical Morrey spaces.

\subsection{Some Known Subspaces}
There is always a need to introduce some known subspaces of Morrey spaces where certain nice properties will hold, like approximation to the identity in the Morrey space does not hold due to the involved function with singularities (example, the homogeneous functions).\\
We will define the vanishing types of Morrey spaces at infinity and at the origin. We have a space called the vanishing at origin Morrey space, denoted $L^{p,\beta}_{0}(\Omega)$ as a predual to $L^{p,\beta}(\Omega)$ consisting of all functions such that
\begin{equation}
\label{eqnvo}
\lim_{r\to 0}\sup_{y\in\Omega}r^{-\beta}\int_{B(y,r)}\vert f(x)\vert^{p}dx=0
\end{equation}
Similarly the condition (\ref{eqnvo}) can be written as
\begin{equation}
\label{eqnvo*}
\sup_{y\in\Omega}r^{-\beta}\int_{B(y,r)}\vert f(x)\vert^{p}dx<\varepsilon\quad\forall\quad r>r_{0}.
\end{equation}
In \cite{11} Chiarenza and Franciosi described the preduals or subspaces satisfying (\ref{eqnvo}) or (\ref{eqnvo*}) as the closed subspace of $L^{p,\beta}$. \\
Recall that
\[\lim_{r\to 0}\sup_{y\in\Omega}r^{-\beta}\int_{B(y,r)}\vert f(x)\vert^{p}dx=0\Leftrightarrow \lim_{r\to 0}\sup_{y\in\Omega;0<r\leq t}r^{-\beta}\int_{B(y,r)}\vert f(x)\vert^{p}dx=0\]
\begin{Rem}
\label{rem2}
Recall that all $L^p$-functions posses some vanishing property where we have that $L_{0}^{p,\beta}=L^p$. However $L_{0}^{p,\beta}$ becomes a proper subset of $L^{p,\beta}$ if $\beta>0$. One of the important features of the above defined subspace not enjoyed by all functions in the Morrey space $L^{p,\beta}(\Omega)$, is the existence of approximation done by bounded functions majorly on bounded domains, as referenced in \cite{11}. This establishes the fact that $L_{0}^{p,\beta}$ is expressed by bounded functions provided the domain is bounded.
\end{Rem}
 \ \\
We also introduced another class of subset of $L^{p,\beta}(\Omega)$ denoted by $\mathbb{L}^{p,\beta}(\Omega)$ introduced in \cite{29}, consisting of all functions $f\in L^{p,\beta}(\Omega)$ whose translation is continuous, that is,
\begin{equation}
\label{eqn2.7}
\mathbb{L}^{p,\beta}(\Omega):=\Big\{\mu \in L^{p,\beta}(\Omega): \Vert\tau_{\eta}\mu-\mu\Vert_{p,\beta}\to 0\quad\text{as}\quad \eta\to 0\Big\}
\end{equation}
where $\tau_{\eta}\mu=\mu(\cdot-\eta), \eta\in\mathbb{R}^n$ (setting $\mu(x)=0$ for $\mathbb{R}^{n}\setminus\Omega$).\\
The space above defined in (\ref{eqn2.7}) is a closed subspace of $L^{p,\beta}$ which is approximated by mollifiers as seen in \cite{29}. Applying similar limiting case when $\beta=0$, the equality is given by $\mathbb{L}^{p,0}(\Omega)=L^{p,0}_{0}(\Omega)=L^{p}(\Omega)$. Let $\beta>0$, the following subspaces are related as follow:
\begin{equation}
\label{eqn2.8}
\mathbb{L}^{p,\beta}(\Omega)\subset L^{p,\beta}_{0}(\Omega) \subset L^{p,\beta}(\Omega)
\end{equation}
Moreover, taking into account for $\vert \Omega\vert<\infty$ (when the domain $\Omega$ is bounded), then
\[\mathbb{L}^{p,\beta}(\Omega)=L^{p,\beta}_{0}(\Omega)\]
The inclusions in (\ref{eqn2.8}) are strict if the underlying domain is unbounded.\\
We introduce another subspace of Morrey spaces satisfying the vanishing condition at infinity as contained in \cite{26*}. This predual is related to the previously defined subspace. Briefly we define another class of subspaces below with some known results.

\begin{Def}
\label{Def3.1}
Let $\beta\in[0,n)$ and $p\in [1,\infty)$. We denote by $L_{\infty}^{p,\infty}$ as a subspace of $L^{p,\beta}(\Omega)$ consisting of all functions $\mu$ such that the following condition is satisfied:
\begin{equation}
\label{eqnv8}
\lim_{r\to\infty}\sup_{y\in\mathbb{R}^{n}}\frac{1}{r^{\beta}}\int_{B(y,r)}\vert f(x)\vert^{p}dx=0
\end{equation}
The space space $L^{p,\beta}_{0}$ is same as the space $L^{p,\beta}_{\infty}(\Omega)$ defined in (\ref{eqnv8}) with a similar condition at infinity for $L^{p,\beta}_{\infty}(\Omega)$.
\end{Def}
\ \\
Another class of subspaces of $L^{p,\beta}$ as a special type of $L_{\infty}^{p,\beta}$ with an additional property of truncation in large balls is introduced, denoted by $L_{\infty,*}^{p,\beta}$.

\begin{Def}
\label{Def3.2}
Let $0\leq \beta<n$ and $p\in [1,\infty)$. The space $L_{\infty,*}^{p,\beta}$ is defined a space consisting of all vanishing functions $f$ such that the following property is fulfilled:
\begin{equation}
\label{eqnv*}
\int_{B(y,r)}\vert f(x)\vert^{p}\chi_{N}(x)dx\leq r^{-\beta}A<\varepsilon\quad\forall\quad N,r>N_{0}\quad\text{and}\quad A\quad\text{is the supremum}
\end{equation}
The above can also written as
\begin{equation}
\label{eqn}
\lim_{r,N\to\infty}\sup_{y\in\mathbb{R}^{n}}\int_{B(y,r)}\vert f(x)\vert^{p}\chi_{N}(x)dx=0
\end{equation}
\end{Def}

From definition \ref{Def3.2}, the truncation in (\ref{eqnv*}) makes it easier for functions to be defined in the uniform Lebesgue space $\mathcal{L}^{p}$ for $1\leq p<\infty$. When functions are not necessarily in the $L^p$-space, the Lebesgue dominated convergence theorem enables us in overcoming difficulties through this theorem.

\subsection{The Spaces $\mathcal{U}$ and $\mathcal{U}'$}
We begin this section with a brief description concerning the spaces $\mathcal{U}$ and $\mathcal{U}'$. We have different versions of generalized functions of great interest, one of which is the ultradistributions of polynomial growth termed the tempered ultradistributions. We adopt the notations used for spaces of rapidly decreasing infinite differentiable functions, spaces of tempered distributions, and the spaces of all distributions with compact support in \cite{28}.
\begin{Def}
\label{defultra}
We define the space $\mathcal{U}(\mathbb{C}^{n})$ as the space of all rapidly decreasing entire functions (or ultradifferentiable functions) $f\in\mathcal{U}\subset\mathcal{S}$ such that following condition is satisfied
\begin{equation}
\label{eqn1}
\Vert f\Vert_{p}=\sup_{\vert\text{Im}(\xi)\vert<p; \xi\in\mathbb{C}}\Big\{(1+\vert \xi\vert^{p})^{\alpha}\vert f(\xi)\vert\Big\}<\infty\quad\forall\quad p\in\mathbb{N}\quad\text{and}\quad \vert\alpha\vert\leq k,\quad k\in\mathbb{R}.
\end{equation}
\end{Def}

\begin{Def}
\label{defdual}
The topological dual of the space $\mathcal{U}$ is called the space of all tempered ultradistributions as introduced by Silva \cite{26}, denoted by $\mathcal{U}'(\mathbb{C}^{n})$.
\end{Def}
\ \\
The continuous linear functional $\mu:\mathcal{U}\longrightarrow\mathbb{C}^{n}$ can be described as the Fourier transform of exponential distributions or as finite order derivatives of bounded functions with exponential growth. The spaces of tempered ultradistributions can be represented by analytical functions. Next we present the vanishing condition and compact supported tempered ultradistributions.

\begin{Def}
\label{defvanish}
A tempered ultradistribution $\mu\in\mathcal{U}'$ is said to vanish in an open subset $\Omega$ if $\displaystyle (1+\vert\xi\vert^{p})^{\alpha}\mu(x+iy)-(1+\vert\xi\vert^{p})^{\alpha}\mu(x-iy)\longrightarrow 0$ as $\vert\xi\vert\to 0$ and $y\to 0$.
\end{Def}

\begin{Def}
\label{defcompact}
A tempered ultradistribution $\mu\in\mathcal{U}'$ has a compact support if there is a representative function $\hat{\mu}$ vanishing at $\infty$. Equivalently, $\mu$ vanishes outside a compact support.
\end{Def}
\ \\
Let $\mathbb{H}\equiv \mathbb{H}(\mathbb{C})$ denotes the space of all entire functions equipped with the topology of uniform convergence on compact subsets in the complex plane. Then we have the following inclusion
\[\mathcal{D}\subset \mathcal{U}\subset \mathcal{S}\subset\mathbb{H}\]
In next section we extend the concept of tempered ultradistributions or ultradistributions of slow growth defined on the space of all rapidly decreasing ultradifferentiable functions to the class of classical Morrey spaces.

\section{Inclusion of Tempered Ultradistributions in Morrey Spaces}
In this section we introduce a tempered Morrey space or slow growth Morrey space, denoted $\mathcal{L}^{p,\beta}_{\mathcal{U}}$ (respectively $L^{p,\beta}_{\mathcal{U}}$) for $1\leq p<\infty$ and $0\leq \beta<n$ by means of growth condition of rapid decrease also satisfying vanishing properties. In particular, the extension of the functions $f\in\mathcal{U}'$ into the Morrey space satisfying the polynomial growth condition is considered. We also show that they are connected with some of the known subspaces of the Morrey spaces introduced in section 2.2. The context here is an open subset $\Omega$ of a complex plane $\mathbb{C}^n$, so that we omit the domain in the notation subsequentlx.\\
We concern with space consisting entirelx of tempered ultradistributions, hence subsets of $L_{\text{loc}}^{1}(\mathbb{C}^{n})\cap \mathcal{U}'(\mathbb{C}^{n})$ since $\mathcal{U}\subset\mathcal{S}$, $\mathcal{S}'\subset\mathcal{U}'$ and $L_{\text{loc}}^{1}(\mathbb{C}^{n})\vert_{\mathbb{R}^{n}}\cap \mathcal{U}'(\mathbb{C}^{n})\vert_{\mathbb{R}^{}}(=\mathcal{S}'(\mathbb{R}^{n}))$. These spaces are developed in the context of Lebesgue-measurable functions.

\begin{Def}
Let $1\leq p<\infty$ and $\alpha$ is a multi-index. A set of all functions $\mu$ belongs to the space $\mathcal{U}'(\mathbb{C}^{n})$, called the tempered ultradistributions, if
\begin{equation}
\label{eqn(*)}
\lim_{\vert\xi\vert\to 0}\sup_{\xi\in\mathbb{C}^{n}}\int_{\mathbb{C}^{n}}(1+\vert\xi\vert^{p})^{\alpha}\vert \mu(\xi)\vert^{p}d\xi=0
\end{equation}
\end{Def}

We introduce the following class of tempered ultradistributons in the classical Morrey space with both global and local versions.

\begin{Def}
\label{def1*}
Let $0<\beta<n$ and $p\in[1,\infty)$ or simply put $0<p\leq \beta<\infty$. We say that a tempered ultradistribution $\mu\in\mathcal{U}'(\mathbb{C}^{n})$ belongs to the global Morrey space $L^{p,\beta}_{\mathcal{U}}$, if there is $\gamma\in\mathcal{U}$ with $\displaystyle \int_{\mathbb{C}^n}\vert\gamma(\xi)\vert d\xi\neq 0$ such that
\[M_{\gamma}\mu(\xi)=\sup_{t\in(0,\infty)}\vert\gamma_{l}*\mu(\xi)\vert=\sup_{\xi\in\mathcal{C}^n}\Big\{(1+\vert\xi\vert^{p})^{\alpha}\vert\gamma_{l}(x-\xi)\mu(\xi)\vert\Big\}<\infty\]
where $M_{\gamma}\mu(\xi)$ is known as the smooth maximal function.
\end{Def}
The functional $\displaystyle\Vert\mu\Vert_{L^{p,\beta}_{\mathcal{U}}}=\Vert M_{\varphi}\mu\Vert_{L^{p,\beta}_{\mathcal{U}}}$ defines an almost norm as $0<p<1$ and becomes a norm for $p\geq 1$. The space $L^{p,\beta}_{\mathcal{U}}$ is defined by the metric $\displaystyle d(f,g):=\Vert f-g\Vert_{L^{p,\beta}_{\mathcal{U}}}$ is a  metric space which is complete and hence a Banach space.\\
Moreover, some properties are inherited from both the space of tempered ultradistributions and Morrey spaces and merged. For example, the convexity in Morrey spaces can be extended by the space $\mathcal{U}'$. Let
\begin{equation}
\label{eqnK}
\mathfrak{K}:=\Big\{\Vert\cdot\Vert_{\eta,t}:\eta,t\in\mathbb{N}_{0}^{n}\quad\text{such that}\quad \vert\eta\vert\leq N',\vert t \vert\leq N'\Big\}
\end{equation}
 be a finite collection of norms defined on $\mathcal{U}(\mathcal{C})$ and consider
\[\mathcal{U}_{\mathfrak{K}}:=\Big\{\gamma\in\mathcal{U}(\mathbb{C}^{n}):\Vert\gamma\Vert_{\eta,t}=\sup_{\xi\in\mathbb{C}}\vert \xi^{\eta}\partial_{\xi}\gamma(\xi)\vert=\sup_{\xi\in\mathbb{C}}\vert (1+\vert\xi\vert^{p})\vert\partial_{\xi}\gamma(\xi)\vert<\infty,\quad\forall\quad\Vert\gamma\Vert_{\eta,t}\in\mathfrak{K}\Big\}.\]

The proof of Theorem \ref{th} is equivalent to the ones present in Lemma 2.1 and Lemma 2.4 in \cite{jia}.
\begin{Th}\cite{jia, almeida}
\label{th}
Given $0<p\leq \beta<\infty$ and $\mu\in\mathcal{U}'(\mathbb{C}^{n})$. Then the following equivalent statements hold:
\begin{itemize}
\item[(i)] There exists $\psi\in\mathcal{U}(\mathbb{C}^{n})$ with $\displaystyle \int_{\mathbb{C}^n}\vert\psi(\xi)\vert d\xi<\infty$ such that $M_{\psi}\mu\in L^{p,\beta}_{\mathcal{U}}$;
\item[(ii)] There is a family $\mathfrak{K}$ defined in (\ref{eqnK}) such that $M_{\mathfrak{K}}\mu\in L^{p,\beta}_{\mathcal{U}}(\mathbb{C})$
\item[(iii)] $\mu$ is a bounded tempered ultradistribution.
\end{itemize}
\end{Th}

We present the definition for the localizable Morrey space below:

\begin{Def}
\label{def1**}
Let $0<\beta<n$ and $p\in [1,\infty)$ or $0<p\leq \beta<\infty$. We say that a tempered ultradistribution $\mu\in\mathcal{U}'(\mathbb{C}^{n})$ belongs to the local Morrey space $\mathcal{L}^{p,\beta}_{\mathcal{U}}$, if there is $\psi\in\mathcal{U}$ with $\displaystyle \int_{\mathbb{C}^n}\vert\psi(\xi)\vert d\xi\neq 0$ such that
\[\sup_{l\in(0,\infty)}\vert\psi_{l}*\mu(\xi)\vert=\sup_{\xi\in\mathcal{C}^n}\Big\{(1+\vert\xi\vert^{p})\vert\psi_{l}(x-\xi)\mu(\xi)\vert\Big\}<\infty\]
\end{Def}
The non-homogeneous version includes the localizable Morrey spaces. That is, we have the continuous inclusion $\displaystyle L^{p,\beta}_{\mathcal{U}}\hookrightarrow\mathcal{L}^{p,\beta}_{\mathcal{U}}$ holds.\\
We also provide alternative definition of the inclusion of Fourier decay estimates for ultradistributions in the Morrey spaces.

  \begin{Def}
  \label{def}
 Let $1\leq p<\infty$, $0<\beta<\infty$ and let $\mu$ be a tempered ultradistribution satisfying condition (\ref{eqn(*)}) on $\mathbb{C}^n$. Then we define the space
 \begin{equation}
 \label{eqn30}
 \mathcal{L}_{\mathcal{U}}^{p,\beta}(\mathbb{C}^{n}):=\Big\{ \mu\in\mathcal{U}\implies \mu\in L_{\text{loc}}^{p}(\mathbb{C}^{n}): \Vert\mu\Vert_{\mathcal{L}_{\mathcal{U}}^{p,\beta}}<\infty\Big\}
\end{equation}
where
\[\Vert\mu\Vert_{\mathcal{L}_{\mathcal{U}}^{p,\beta}}=\sup_{\xi\in\mathbb{C}^{n}}\Big(\sup_{y\in\Omega;r>0}\frac{1}{r^\beta}\int_{\tilde{B}(y,r)}(1+\vert\xi\vert^{p})^{\alpha}\vert \mu(\xi)\vert^{p}d\xi\Big)\]
and the double supremum is taken over a complex function $\xi$ and taken over a ball $\bar{B}$ in $\mathbb{C}^{n}$.
  \end{Def}

The following result helps us in transferring properties from the global tempered Morrey space $\displaystyle L^{p,\beta}_{\mathcal{U}}$ to the corresponding local tempered Morrey space $\mathcal{L}^{p,\beta}_{\mathcal{U}}$, all satisfying the growth condition.

\begin{Lem}
\label{lem}
Let $\gamma\in\mathcal{U}$, $\displaystyle \int_{\mathcal{C}^{n}}\vert\gamma(\xi)\vert d\xi=1$ and $\displaystyle \int_{\mathcal{C}^{n}}(1+\vert\xi\vert^{p})^{\alpha}\vert\gamma(\xi)\vert d\xi=0\quad\forall\quad$ multi-index $\alpha\in\mathbb{N}_{0}^{n}$ such that $\vert\alpha\vert\neq 0$. Then there is $M>0$ such that
\[\Vert\mu-\gamma*\mu\Vert_{L^{p,\beta}}\leq M\Vert\mu\Vert_{\mathcal{L}^{p,\beta}}\]
for every $\mu\in\mathcal{L}^{p,\beta}_{\mathcal{U}}$ with $0<p<1$ and $p\leq \beta<\infty$.
\end{Lem}

\begin{proof}
Let $\mu\in\mathcal{L}^{p,\beta}_{\mathcal{U}}$, and let $\displaystyle m_{\mathfrak{K}}\mu(\xi)=\sup_{l\in(0,\infty)}\vert\psi_{l}*\mu(\xi)\vert=\sup_{\xi\in\mathcal{C}^n}\Big\{(1+\vert\xi\vert^{p})\vert\psi_{l}(x-\xi)\mu(\xi)\vert\Big\}$ such that it is finite. Then we have
\begin{align*}
\Vert\mu-\gamma*\mu\Vert_{L^{p,\beta}_{\mathcal{U}}} & \lesssim \Vert\sup_{0<l<1}\vert\varphi_{l}*\mu\vert\Vert_{\mathcal{L}^{p,\beta}_{\mathcal{U}}} +\Vert\sup_{0<l<1}\vert(\varphi_{l}*\gamma)*\mu\vert\Vert_{\mathcal{L}^{p,\beta}_{\mathcal{U}}}\\
    &  +\Vert\sup_{l>1}\vert(\varphi_{l}-\varphi_{l}*\gamma)*\mu\vert\Vert_{L^{p,\beta}_{\mathcal{U}}}=\Vert\mu\Vert_{\mathcal{L}^{p,\beta}_{\mathcal{U}}}+I_{1}(\mu)+I_{2}(\mu)
\end{align*}
Since $\displaystyle \Big\{ \varphi_{l}*\gamma\Big\}$ and $\displaystyle \Big\{ \varphi_{l}-\varphi_{l}*\gamma\Big\}$ are bounded in $\mathcal{U}(\mathbb{C})$, then
\[\vert (\varphi_{l}*\gamma)*\mu\vert\lesssim m_{\mathfrak{K}}\mu\quad\text{for}\quad l\leq 1\]
and
\[\vert (\varphi_{l}-\varphi_{l}*\gamma)*\mu\vert\lesssim m_{\mathfrak{K}}\mu\quad\text{for}\quad l>1\]
which gives
\[I_{i}(\mu)\lesssim \Vert m_{\mathfrak{K}}\mu\Vert_{L^{p,\beta}_{\mathcal{U}}}\lesssim \Vert m_{\mathfrak{K}}\mu\Vert_{\mathcal{L}^{p,\beta}_{\mathcal{U}}}\]
for any choice of $i$ from 1.
\end{proof}

\subsection{Some Properties of Morrey Spaces}
We present some important results concerning the inclusion of tempered ultradistribution in the Morrey space, and we give some of the results involving some subspaces of Morrey spaces.
\begin{Lem}
\label{lemq}
Given $p'\leq \beta<\infty$, $p',\beta>0$ and $p\leq \gamma<\infty$, $p,\gamma>0$ with $p\beta=p'\gamma$. If $\mu\in L^{p',\beta}_{\mathcal{U}}$ and $\beta\leq \gamma$ then there is a positive constant $M$ (not dependent on $\mu$) such that
\[\Vert\varphi_{l}*\mu\Vert_{L^{p,\beta}_{\mathcal{U}}}\leq M l^{n\Big(\frac{1}{\gamma}-\frac{1}{\beta}\Big)}\Vert\mu\Vert_{L^{p',\beta}_{\mathcal{U}}}\]
provided $\varphi\in\mathcal{U}(\mathbb{C}^n)$ and $\displaystyle \int_{\mathbb{C}^{n}}\vert\varphi(\xi)\vert d\xi\neq 0$.
\end{Lem}

The proof of the claim in Lemma \ref{lemq} can be found in Proposition \ref{prop}.

\begin{Prop}
\label{prop}
Let $0<p'\leq 1$ and $0\leq \beta<\infty$. Any bounded function with compact support which satisfies the following condition
\[\frac{1}{r^\beta}\int_{\mathbb{C}^{n}}(1+\vert\xi\vert^{p})^{\alpha}\vert \mu(\xi)\vert d\xi=0\quad\text{for}\quad \vert\alpha\vert\leq K\]
with $\displaystyle K\geq N_{p'}:=\lfloor \Big(\frac{1-p'}{p}\Big)n\rfloor$ belongs to $L^{q,\beta}_{\mathcal{U}}$. \\
Furthermore
\[\Vert\mu\Vert_{L^{p',\beta}_{\mathcal{U}}}\lesssim \Vert \mu\Vert_{L^{\infty}}\vert\Omega\vert^{\frac{1}{\beta}}\]
$\forall$ $\Omega$ such that $\text{supp}(\mu)\subset \Omega$.
\end{Prop}

We present the following result as related to the space defined in (\ref{eqn30}).

\begin{Lem}
\label{lem}
A function $f\in\mathcal{L}_{\mathcal{U}}^{p,\beta}$ satisfies condition (\ref{eqn(*)}) if and only if
\begin{equation}
\label{eqn3.1}
\lim_{\vert\xi\vert\to\infty}\sup_{\xi\in\mathbb{C}^{n}}\sup_{y\in\Omega;r>0}\frac{1}{r^\beta}\int_{B(y,r)}(1+\vert\xi\vert^{p})^{\alpha}\vert \mu(\xi)\vert^{p}d\xi=0
\end{equation}
uniformly (absolutely) in $0<r\leq l$ for fixed $l>0$.
\end{Lem}
  \begin{proof}
  Assume condition (\ref{eqn3.1}), then clearly (\ref{eqn3.1}) implies condition (\ref{eqn(*)}). Conversely, let assume condition (\ref{eqn(*)}) we prove that (\ref{eqn3.1}) is satisfied. To see this, suppose that $\mu$ satisfies (\ref{eqn(*)}). Let $l>0$ be arbitrarily fixed and $y\in\Omega$. Then there is $m\in\mathbb{N}$ (dependent on $l$ and $n$) and $y_{i}\in B(y,t),i=1,\cdots,m$ such that
  \[B(y,r)\subset \bigcup_{i\in \mathbb{N}^{m}}B(y_{i},1)\]
  $\forall$ $r\in(0,l]$. But
  \begin{align*}
  \int_{B(y,r)}(1+\vert\xi\vert^{p})^{\alpha}\vert \mu(\xi)\vert^{p}d\xi &\leq \sum_{i=1}^{m} \int_{B(y_{i},1)}(1+\vert\xi\vert^{p})^{\alpha}\vert \mu(\xi)\vert^{p}d\xi\\
        & \leq M\sup_{\xi\in\mathbb{C}^{n}}\int_{B(y_{i},1)}(1+\vert\xi\vert^{p})^{\alpha}\vert \mu(\xi)\vert^{p}d\xi\\
        &\leq M \sup_{\xi\in\mathbb{C}^{n}}\sup_{y\in\Omega;r>0}\frac{1}{r^\beta}\int_{B(y,r)}(1+\vert\xi\vert^{p})^{\alpha}\vert \mu(\xi)\vert^{p}d\xi
  \end{align*}
For a constant $M>0$ independent of $y$ and $r$, we get
  \[\sup_{\xi\in\mathbb{C}^{n}}\sup_{y\in\Omega;r>0}\frac{1}{r^\beta}\int_{B(y,r)}(1+\vert\xi\vert^{p})^{\alpha}\vert \mu(\xi)\vert^{p}d\xi\leq M \sup_{y\in\Omega;r>0}\frac{1}{r^\beta}\int_{B(y,r)}\vert \mu(\xi)\vert^{p}d\xi\]
where $\displaystyle M=\sup_{\xi\in\mathbb{C}^{n}}\Big\{(1+\vert\xi\vert^{p})^{\alpha}\Big\}$ from which (\ref{eqn3.1}) follows, uniformly in $0<r\leq l$ for $l>0$.
  \end{proof}

  \begin{Def}
  \label{Def3.5}
  Let $0\leq \beta<n$ and $1\leq p<\infty$, we define the subspace $L^{p,\beta}_{0,\infty,*}$ of $L^{p,\beta}$ (respectively $\mathcal{L}^{p,\beta}_{0,\infty,*}$ of $\mathcal{L}^{p,\beta}$) as
  \[L^{p,\beta}_{0,\infty,*}=L^{p,\beta}_{0}\cap L^{p,\beta}_{\infty}\cap L^{p,\beta}_{*}\]
  Similarly, we have
  \[\mathcal{L}^{p,\beta}_{0,\infty,*}=\mathcal{L}^{p,\beta}_{0}\cap \mathcal{L}^{p,\beta}_{\infty}\cap \mathcal{L}^{p,\beta}_{*}\]
  \end{Def}

\begin{Rem}
We can define the function $\psi$ by
\[\psi(y)=\sum_{k=2}^{\infty}\chi_{B_{t}(y)}\]
where $B_{t}=B(2^{-t}e_{1},1)$, $t\in\mathbb{N}$, with $e_{1}=(1,0,\cdots,0)$. Then $\psi\in L^{p,\beta}_{0}\bigcap L^{p,\beta}_{\infty}$ but $\psi\notin L^{p,\beta}_{*}$
\end{Rem}

Consider the following result as related to the already defined subspaces of Morrey spaces.

\begin{Rem}
\label{Rem3.7}
Clearly, all the vanishing subsets $L^{p,\beta}_{0}, L^{p,\beta}_{\infty}$, $L^{p,\beta}_{*}$ and $L^{p,\beta}_{0,\infty}$ are closed in $L^{p,\beta}$
\end{Rem}

\begin{Lem}
\label{lem3.8}
The space $\mathcal{L}^{p,\beta}_{\mathcal{U}}$ is closed in $L^{p,\beta}$.
\end{Lem}
\begin{proof}
We show that $\mathcal{L}^{p,\beta}_{\mathcal{U}}$ is closed in $L^{p,\beta}$ since it contains rapidly decreasing ultradifferentiable functions on $\mathbb{C}^n$. To do this, denote $(\mu_{m})_{m}$ as a sequence of tempered ultradistributions in $\mathcal{L}^{p,\beta}_{\mathcal{U}}$ converging to $\mu$ in $L^{p,\beta}$. For any $\varepsilon>0$ there exists $\bar{k}\in\mathbb{N}$ such that
\[(1+\vert\xi\vert^{p})^{\alpha}(\mu-\mu_{\bar{k}})\leq \Vert\mu-\mu_{\bar{k}}\Vert_{p,\beta}^{p}<\frac{\varepsilon}{2^{p}2^{q}}\quad\text{for}\quad q>0\]
Hence we have from
\[\sup_{\xi\in\mathbb{C}^{n}}(1+\vert\xi\vert^{p})^{\alpha}\vert\mu(\xi)\vert<\infty\]
Then $\forall$ $\bar{n}\geq n_{0}$
\[\int_{B(y,r)}(1+\vert\xi\vert^{p})^{\alpha}\vert\mu(\xi)\vert^{p}d\xi\leq 2^{p}\int_{B(y,r)}(1+\vert\xi\vert^{p})^{\alpha}\vert\mu(\xi)-\mu_{\bar{k}}(\xi)\vert^{p}d\xi+2^{q}\int_{B(y,r)}(1+\vert\xi\vert^{p})^{\alpha}\vert\mu_{\bar{k}}(\xi)\vert^{p}d\xi<\varepsilon\]
for arbitrary large $\bar{n}$.
\end{proof}

We present the preservation of polynomial growth condition as related to the convolution operator.
\begin{Th}
\label{th3.8}
Let $0\leq \beta\leq n$ and $1\leq p<\infty$. The preduals  $L^{p,\beta}_{0}, L^{p,\beta}_{\infty}, L^{p,\beta}_{*}$ and the dual space $\mathcal{L}^{p,\beta}_{\mathcal{U}}$ are invariant with respect to convolution with integrable operator.
\end{Th}

\begin{proof}
Let $\Phi\in L^1$ and $\mu\in L^{p,\beta}$ then $\mu*\Phi\in L^{p,\beta}$. Let $y\in\Omega$ and $r>0$. Then by Holder's inequality, we get
\begin{align*}
\frac{1}{r^{\beta}}\int_{B(y,r)}\vert (\mu*\Phi)(x)\vert^{p}dx & \leq \frac{1}{r^{\beta}}\Big[\int_{\mathbb{C}^n}\Big(\int_{B(y,r)}\vert \mu(x-\xi)\phi(\xi)\vert^{p}dx\Big)^{\frac{1}{p}}d\xi\Big]^{p}\\
      &=\frac{1}{r^{\beta}}\Big[\int_{\mathbb{C}^n}\vert\Phi(\xi)\vert^{p} \Big(\int_{B(y-\xi,r)}\vert \mu(u)\vert^{p}du\Big)^{\frac{1}{p}}d\xi\Big]^{p}\\
   \leq \Vert\Phi\Vert_{1}^{p}\sup_{\psi\in\mathbb{C}^{n}}r^{-\beta}\int_{B}\vert\mu(u)\vert^{p}du
\end{align*}
It shows that $\mu\in L^{p,\beta}_{0}\implies \mu*\Phi\in L^{p,\beta}_{0}$ and $\mu\in L^{p,\beta}_{\infty}\implies \mu*\Phi\in L^{p,\beta}_{\infty}$ from Theorem 3.8 \cite{26*}.\\
Next we consider the preservation of condition (\ref{eqn30}) by convolution operator. Let $\mu\in\mathcal{L}_{\mathcal{U}}^{p,\beta}$, we get
\begin{equation}
\label{eqnconvolution}
\Big[\int_{B(0,r)}(1+\vert\xi\vert^{p})^{\alpha}\vert(\mu*\Phi)(\xi)\vert^{p}d\xi\Big]^{\frac{1}{p}}\leq \int_{\mathbb{C}^{n}}\vert\Phi(z)\vert^{p} \Big(\int_{B(0,r)}(1+\vert\xi\vert^{p})^{\alpha}\vert f(\xi-z)\vert^{p}d\xi\Big)^{\frac{1}{p}}dz
\end{equation}
From the change of variables on the right-hand side of (\ref{eqnconvolution}). we get
\[\Big[\int_{B(0,r)}(1+\vert\xi\vert^{p})^{\alpha}\vert(\mu*\Phi)(\xi)\vert^{p}d\xi\Big]^{\frac{1}{p}}\leq \int_{\mathbb{C}^{n}}\vert\Phi(z)\vert^{p} \Big(\sup_{\psi\in\mathbb{C}^{n}}\int_{B(\psi,r)}(1+\vert u+z\vert^{k})^{p}\vert \mu(u)\vert^{p}du\Big)^{\frac{1}{p}}dz\]
Since this is uniform with respect to the origin, we obtain
\begin{align*}
\frac{1}{r^\beta}\int_{B(0,r)}(1+\vert\xi\vert^{p})^{\alpha}\vert(\mu*\Phi)(\xi)\vert^{p}d\xi & \leq \frac{1}{r^\beta} \Big[\int_{\mathbb{C}^{n}}\Big(\int_{B(0,r)}(1+\vert\xi\vert^{p})^{\alpha}\vert \mu(\xi-z)\Phi(z)\vert^{p}d\xi\Big)^{\frac{1}{p}}dz\Big]^{p}\\
      &\leq \sup_{r>0}\frac{1}{r^\beta} \Big[\int_{\mathbb{C}^{n}}\Big(\int_{B(0,r)}(1+\vert\xi\vert^{p})^{\alpha}\vert \mu(\xi-z)\Phi(z)\vert^{p}d\xi\Big)^{\frac{1}{p}}dz\Big]^{p}\\
      &\leq \sup_{\xi\in\mathbb{C}^{n}}\sup_{r>0}\frac{1}{r^\beta} \Big[\int_{\mathbb{C}^{n}}\Big(\int_{B(0,r)}(1+\vert\xi\vert^{p})^{\alpha}\vert \mu(\xi-z)\Phi(z)\vert^{p}d\xi\Big)^{\frac{1}{p}}dz\Big]^{p}\\
      &\leq \sup_{u+z\in\mathbb{C}^{n}}\int_{\mathbb{C}^{n}}\vert\Phi(z)\vert^{p} \Big(\sup_{\psi\in\mathbb{C}^{n}}\int_{B(\psi,r)}(1+\vert u+z\vert^{p})^{\alpha}\vert \mu(u)\vert^{p}du\Big)^{\frac{1}{p}}dz
\end{align*}
This shows that $(1+\vert\xi\vert^{p})^{\alpha}(\mu*\Phi)(\xi)$ converges to zero as $\vert\xi\vert\to\infty$.
\end{proof}

\section{Embeddings in Morrey Spaces}
Here we are concerned with the relation between different versions of Morrey spaces or the relationship between the tempered Morrey spaces, Morrey pre-duals and spaces of tempered ultradistributions. We dwell with questions about embedding of related spaces under consideration.\\
We begin with some known results on strict embeddings.

\begin{Th}
\label{th4.1}
Let  $\beta\in (0,n)$ and $p\in [1,\infty)$. Then there exist functions in $L^{p,\beta}_{0}\bigcap L^{p,\beta}_{\infty}$ in which the property (\ref{eqnv*}) fails.
\end{Th}

\begin{Rem}
 In Remark \ref{rem*},  emphasis was laid on the roles of the parameters $p$ and $\beta$ in the tempered classical Morrey spaces both in local and global versions. In what follows, we present the extension of the scaling law or the multiplication law from the classical Morrey space to the tempered classical Morrey spaces for all $0\leq \beta<n$ and $1\leq p<\infty$.
\end{Rem}

\begin{Th}
\label{thmscaling}
Given $1\leq p<\infty$ and $0\leq \beta<n$ for $n\in\mathbb{N}$. Then for all $\mu\in L_{\mathcal{U}}^{p,\beta}$ and $l>0$, the following property holds:
\begin{equation}
\label{eqnscaling}
\Vert\mu(l\cdot)\Vert_{L_{\mathcal{U}}^{p,\beta}}\leq l^{-n/p}\Vert\mu\Vert_{L_{\mathcal{U}}^{p,\beta}}
\end{equation}
\end{Th}

\begin{proof}
Let $l,\xi\in\mathbb{C}^n$ then $l\xi\in\mathbb{C}^n$. Thus we have

\begin{align*}
\Vert\mu\Vert_{\mathcal{L}_{\mathcal{U}}^{p,\beta}}& =\sup_{\xi\in\mathbb{C};y\in\Omega;r>0}\frac{1}{r^\beta}\Big(\int_{\bar{B}(y,r)}(1+\vert\xi\vert^{p})^{\alpha}\vert\mu(l\xi)\vert^{p}d\xi\Big)^{\frac{1}{p}}\\
   &=\sup_{\xi\in\mathbb{C};y\in\Omega;r>0}\frac{1}{r^\beta}\Big(\frac{1}{(lr)^{n}}\int_{\bar{B}(ly,lr)}(1+\vert\xi\vert^{p})^{\alpha}\vert\mu(\xi)\vert^{p}d\xi\Big)^{\frac{1}{p}}\\
     &=\sup_{\xi\in\mathbb{C};y\in\Omega;r>0}\frac{1}{r^\beta}\Big(\frac{1}{l^{n}r^{n}}\int_{\bar{B}(ly,lr)}(1+\vert\xi\vert^{p})^{\alpha}\vert\mu(\xi)\vert^{p}d\xi\Big)^{\frac{1}{p}}\\
    &= l^{-n/p}\sup_{\xi\in\mathbb{C};y\in\Omega;r>0}\frac{1}{r^{\beta+n}}\Big(\int_{\bar{B}(ly,lr)}(1+\vert\xi\vert^{p})^{\alpha}\vert\mu(\xi)\vert^{p}d\xi\Big)^{\frac{1}{p}}
\end{align*}
Since $\displaystyle \Big\{ (1+\vert\xi\vert^{p})^{\alpha}\vert\mu(\xi)\vert \Big\}<\infty$. Hence the result follows.
\end{proof}

It is important to show the local integrability of tempered ultradistributions in the classical Morrey space via the parameter $\beta$ as can be seen in the next result:

\begin{Th}
\label{thmlocalintegral}
Let $0<p<p'<\beta<\infty$. Then $\displaystyle L_{\mathcal{U}}^{p,\beta}\cap L_{c}^{0}\cap L_{0,\infty}^{p,\beta}\setminus L^{\text{loc}}_{p',\beta}\neq \emptyset.$
\end{Th}

\begin{proof}
Since $p<p'<\beta$, it implies that $\Vert\cdot\Vert_{p'}\leq \Vert\cdot\Vert_{p}$. We prove that $\displaystyle \mu \in L_{\mathcal{U}}^{p,\beta}\cap L_{c}^{0}\cap L_{0,\infty}^{p,\beta} $ is not empty provided $\mu \notin L^{\text{loc}}_{p',\beta}$, it shows that $\mu$ must have a compact, and must satisfy both vanishing and slow growth condition. To see this, since $\displaystyle \Big\{ (1+\vert\xi\vert^{p})^{\alpha}\vert\mu(x)\vert \Big\}<\infty$, using an example of a function to back up our claim as $\mu\equiv\sum\mu\chi_{B(0,l)}$, we see that $\mu\in L_{0,\infty}^{p,\beta}$ for as $l\to \infty$ and $l\to 0$. Similarly $\mu\in L_{c}^{0}$ since it has a support in $\bar{B}(0,l)$. Then take the $\liminf$ of the function $\mu$, we have that $\displaystyle \liminf\Big( \frac{\Vert\mu\Vert_{L^{p'}_{\text{loc}}}}{\Vert\mu\Vert_{L_{\mathcal{U}}^{p,\beta}}}\Big)<\infty$. Hence the proof.
\end{proof}

\begin{Cor}
\label{cor4.2}
Let $\beta\in (0,n)$ and $p\in [1,\infty)$. Then we have $\displaystyle L^{p,\beta}_{0,\infty,*}\subsetneqq L^{p,\beta}_{0}\bigcap L^{p,\beta}_{\infty}\subsetneqq L^{p,\beta}_{0} \subsetneqq L^{p,\beta}$
\end{Cor}
\ \\
Corollary \ref{cor4.2} shows that the predual $L_{0,\infty,*}^{p,\beta}$ is strictly smaller than $\mathbb{L}^{p,\beta}$ as presented in Theorem \ref{th4.3}. That is, $\displaystyle L_{0,\infty,*}^{p,\beta}\subset \mathbb{L}^{p,\beta}$. The proof of Corollary \ref{cor4.2} makes use of the vanishing properties and approximation theory by functions with compact support.

\begin{Th}
\label{th4.3}
Given $\beta\in [0,n)$ and $p\in[1,\infty)$ for $n\in\mathbb{N}$. Then the following inclusion holds: $\displaystyle L^{p,\beta}_{0,\infty,*} \subset \mathbb{L}^{p,\beta}$.
\end{Th}

\begin{Rem}
The continuous embedding $X\hookrightarrow Y$ is also as a result of immediate consequence of triangle inequality. As a consequence of Theorem 3.1 in \cite{22}, $\mathcal{D}\subset \mathcal{U}\subset\mathcal{S}\subset\mathcal{E}$, where $\mathcal{D}$ and $\mathcal{E}$ are spaces of test functions with compact support and exponential type respectively, and $\mathcal{E}'\subset \mathcal{S}'\subset\mathcal{U}'\subset\mathcal{D}'$, we present the result on $\mathcal{L}^{p,\beta}_{\mathcal{U}}$ and $L^{p,\beta}$ by the approximation via tempered ultradistributions
\end{Rem}

\begin{Th}
\label{th3.1}
Given $p\in (1,\infty)$, $-\frac{n}{p}<\beta<0$ and $\beta =-n/u$. Then the spaces $\mathcal{L}^{p,\beta}_{\mathcal{U}}$ and $L^{p,\beta}_{\mathcal{U}}$ are Banach spaces, for $1/p<u/p<1$ and
\begin{equation}
\label{eqn3.5}
\mathcal{U}(\mathbb{C}^{n})\hookrightarrow L^{p} \hookrightarrow \mathcal{L}^{p,\beta}_{\mathcal{U}} \hookrightarrow L^{p,\beta}_{\mathcal{U}}\hookrightarrow L^{u}_{\mathcal{U}}
\end{equation}
Moreover, the spaces $\mathcal{D}, \mathcal{U},\mathcal{S}$ and $L^p$ are dense in $\mathcal{L}^{p,\beta}_{\mathcal{U}}$ and $L^{p,\beta}_{\mathcal{U}}$
\end{Th}

\begin{proof}
For $1<u<p$ then (\ref{eqn3.5}) is obvious. Note that the continuous embedding $L^{p,\beta}(\mathbb{R}^{n})\hookrightarrow \mathcal{L}^{p,\beta}(\mathbb{R}^{n})$ is not complete. The spaces $L^{p,\beta}_{\mathcal{U}}(\mathbb{C}^{n})$ and  $\mathcal{L}^{p,\beta}_{\mathcal{U}}(\mathbb{C}^{n})$ are Banach.\\
The embedding of $\mathcal{L}^{p,\beta}_{\mathcal{U}}(\mathbb{C}^{n})$ into $L^{p,\beta}_{\mathcal{U}}(\mathbb{C}^{n})$ is obvious be definition since the local Morrey space is embedded into global Morrey space. We now prove the last embedding. To do this, let $\mu\in L^{p,\beta}_{\mathcal{U}}$. Let
\begin{equation}
\label{eqn3.9}
\frac{1}{u}=\frac{1}{p}+\frac{1}{p'},\quad \text{where}\quad \frac{1}{p'}=-\frac{1}{p}-\frac{\beta}{n}
\end{equation}
Then we have
\begin{align*}
\Vert\mu\Vert_{L^{u}_{\mathcal{U}}} & \leq M\sum_{\xi\in \mathbb{C}^n}(1+\vert\xi\vert^{p})^{\alpha}\Vert\mu\Vert_{L^{p,\beta}_{\mathbb{U}}} \\
      &\leq M\sup_{\xi\in \mathbb{C}^n}\Vert\mu\Vert_{L^{p,\beta}_{\mathcal{U}}}
\end{align*}
This shows that $\displaystyle L^{p,\beta}_{\mathcal{U}}\hookrightarrow L^{u}_{\mathcal{U}}$.\\
Finally we prove that $\mathcal{D}$ is dense in both $\mathcal{L}_{\mathcal{U}}^{p,\beta}$ and $L^{p,\beta}_{\mathcal{U}}$. Let $\hat{\mu}\in \mathcal{U}'(\mathbb{C}^{n})$ we get
\[\hat{\mu}^{\alpha}=\sum_{\vert h\vert\leq \alpha, \vert k\vert\leq \alpha}(1+\vert\xi\vert^{p})^{\alpha}\hat{\mu}_{h,k},\quad \alpha=(\alpha_{1},\cdots,\alpha_{n})\in \mathbb{N}\]
Then
\[\Vert \hat{\mu}-\hat{\mu}^{\alpha}\Vert_{\mathcal{L}^{p,\beta}_{\mathcal{U}}}\longrightarrow 0\quad\text{as}\quad \vert\alpha\vert\to \infty\]
Any $\vert h\vert\leq \alpha, \vert k\vert\leq \alpha$ can be represented in $L^{p,\beta}$ by function belonging to $\mathcal{D}$. Then $\mathcal{D}$ is dense in $L^{p,\beta}_{\mathcal{U}}$. By similar argument, one has that $\mathcal{D}$ is dense in $\mathcal{L}_{\mathcal{U}}^{p,\beta}$.
\end{proof}

\begin{Th}
\label{th*}
For $n\in\mathbb{N}$, let $0\leq \beta<n$ and $1\leq p<\infty$. Then we have the following inclusion: $\displaystyle \mathcal{L}_{\mathcal{U}}^{p,\beta}\subset L^{p,\beta}_{0,\infty,*}\subset \mathbb{L}^{p,\beta}$.
\end{Th}
\begin{proof}
The proof of $L^{p,\beta}_{0,\infty,*}\subset \mathbb{L}^{p,\beta}$ can be found in (Theorem 4.3 \cite{26*}) . We show that $\displaystyle \mathcal{L}_{\mathcal{U}}^{p,\beta}\subset \mathbb{L}^{p,\beta}$. Let $\mu\in \mathcal{L}_{\mathcal{U}}^{p,\beta}$. Since $\displaystyle \mathcal{L}_{\mathcal{U}}^{p,\beta}\subset L^{p,\beta}$ then $\forall$ $\varepsilon>0$ there is a $\gamma\in L_{\mathcal{U}}^{p,\beta}$ with compact support such that
\[\Vert\mu-\gamma\Vert_{p,\beta}<\frac{\varepsilon}{4Q}\quad\text{and}\quad \Vert\tau\psi-\psi\Vert_{p,\beta}<\frac{\varepsilon}{2Q}\quad\text{for}\quad Q>0\]
For any $\xi\in\mathbb{C}^n$ we have
\[\Vert\tau_{\varepsilon}\mu-\mu\Vert_{p,\beta}\leq 2Q \Vert\mu-\gamma\Vert_{p,\beta}+\Vert\tau_{\varepsilon}\gamma-\gamma\Vert_{p,\beta} \leq 2\Big(\frac{\varepsilon}{4Q}\Big)+\frac{\varepsilon}{2Q}< \frac{\varepsilon}{Q}<\varepsilon\]
Thus $\displaystyle \Vert\tau_{\varepsilon}\mu-\mu\Vert_{p,\beta}< \varepsilon$. In view of conditions of (\ref{eqn(*)}) and (\ref{eqn30}), there are $r,l>0$ we have
\begin{align*}
C_{1}\equiv\sup_{\xi\in\mathbb{C}^{n}}\sup_{y\in\Omega; r\in (0,l)}(1+\vert \xi\vert^{p})^{\alpha}r^{-\beta}\int_{B(y,r)}\vert\tau\psi-\psi(x)\vert^{p}dx & \\
 \leq M 2 ^{p}\sup_{y\in\Omega; r\in (0,l}(1+\vert \xi\vert^{p})^{\alpha}r^{-\beta}\int_{B(y,r)}\vert\psi(x)\vert^{p}dx<M\Big(\frac{\varepsilon}{2}\Big)^{p}
\end{align*}
and
\begin{align*}
C_{2}\equiv\sup_{\xi\in\mathbb{C}^{n}}\sup_{y\in\Omega; r>l}(1+\vert \xi\vert^{p})^{\alpha}r^{-\beta}\int_{B(y,r)}\vert\tau\psi-\psi(x)\vert^{p}dx & \\
leq M 2 ^{p}\sup_{y\in\Omega; r>l}(1+\vert \xi\vert^{p})^{\alpha}r^{-\beta}\int_{B(y,r)}\vert\psi(x)\vert^{p}dx<M_{1}\Big(\frac{\varepsilon}{2}\Big)^{p}
\end{align*}

 By fixing $l>0$ and $r_0$ we have
 \[\Vert\tau_{\varepsilon}\gamma-\gamma\Vert_{p,\beta}\leq Q\max\lbrace C_{1},C_{2},C_{3}\rbrace\]
 with
 \[C_{3}:=\sup_{y\in\Omega; r\in[r_{0},t]}(1+\vert \xi\vert^{p})^{\alpha}r^{-\beta}\int_{B(y,r}\vert\tau\psi-\psi(x)\vert^{p}dx\]
 Thus
 \[C_{3}\lesssim\sup_{\xi\in\mathbb{C}^{n}}\sup_{y\in\Omega; r}\int_{B(y,t)}(1+\vert \xi\vert^{p})^{\alpha}\vert\psi(x-\xi)-\psi(x)\vert^{p}d\xi\leq N\max\Big\{ \sup_{\vert y\vert\leq M}(\cdots),\sup_{\vert y\vert>M}(\cdots)\Big\}\]
 with $N=M_{1}M>0$ and $N,\vert\alpha\vert>0$. Since $\psi$ has a compact support in $\bar{B}(y,r)$, we have $T>0$ such that $\psi(\cdot)=0$ if $\text{supp}(\psi)\subseteq \Omega$ and $\psi(\cdot)\neq 0$ if $\text{supp}(\psi)\subset \Omega$. In the case $\vert\xi\vert>M$ we have
 \[\int_{\vert y-x\vert<l}(1+\vert \xi\vert^{p})^{\alpha}\vert\gamma(y-\xi)-\gamma(x)\vert^{p}dx=\int_{\vert z\vert<l}\vert\gamma(z+y-\xi)-\gamma(z+y)\vert^{p}dz\]
Thus we obtain
 \[C_{\varepsilon}<\Big(\frac{\varepsilon}{2Q}\Big)<\varepsilon\]
with $L^p$-norm continuity with respect to translation. Hence the proof.
\end{proof}

\begin{Rem}
\label{rem4.7}
With the consideration of proof of Theorem \ref{th*}, we obtain that for every $\varepsilon>0$ and  $\mu\in L^{p,\beta}_{0}\cap \mathcal{L}^{p,\beta}_{\mathcal{U}}\cap L^{p,\beta}_{\infty}$, there exist  $r_{0},l>0$ such that
\[\Vert\tau_{\varepsilon}\mu-\mu\Vert_{p,\beta}^{p}\leq N\sup\lbrace \varepsilon,C_{r_{0},t(\xi)}\rbrace\]
$\forall$ $\xi\in\mathbb{C}^n$, $C_{r_{0},l}(\xi)$ is expressed
\[C_{r_{0},l}(\xi)=\sup_{\xi\in\mathbb{C}^{n}}\Big[\sup_{y\in\mathbb{R}^{n},r>0}\int_{B(y,l)}(1+\vert\xi\vert^{p})^{\alpha}\vert\mu(y-\xi)-\mu(\xi)\vert^{p}d\xi\Big]\]
Note that the rapidly decreasing condition ensures the uniform continuity of the translation operator in $\mathcal{L}^{p,\beta}_{\mathcal{U}}$, and controls the vanishing condition for Morrey tempered ultradistributions in $\displaystyle L^{p,\beta}_{0}\cap\mathcal{L}^{p,\beta}_{\mathcal{U}}\cap L^{p,\beta}_{\infty}$. If we take rapidly decreasing ultradifferentiable functions, say $\psi\in\mathcal{U}$, then the mollifier $\mu*\psi_{l}\in \mathbb{L}^{p,\beta}\cap C^{\infty}$ $\forall$ $l>0$.
\end{Rem}
In view of Remark \ref{rem4.7} we present the following theorem.

\begin{Th}
\label{th5.1}
Assume $0\leq \beta<n$ for $n\in\mathbb{N}$ and $1\leq p<\infty$. Then for every Morrey function with vanishing condition in Morrey norm there exists a $C^\infty$ function. Furthermore, we have
\[\mu\in\overline{\mathbb{L}^{p,\beta}\cap C^{\infty}}\implies \mu\in \mathbb{L}^{p,\beta}.\]
\end{Th}
In view of all the results presented we have the following important theorem which characterizes the functions in the classical Morrey space with both vanishing condition and slow growth condition.

\begin{Th}
\label{th5.2}
For $\beta\in [0,n)$ and $p\in [1,\infty)$,  $n\in \mathbb{N}$. If $\mu\in L^{p,\beta}_{0}\cap \mathcal{L}_{\mathcal{U}}^{p,\beta}\cap L^{p,\beta}_{\infty}$ is uniformly continuous. Then $\mu$ can be represented in Morrey classical norm by function from $\displaystyle L^{p,\beta}_{0}\cap \mathcal{L}_{\mathcal{U}}^{p,\beta}\cap L^{p,\beta}_{\infty}\cap C^\infty$.
\end{Th}

\begin{proof}
We show that any function that satisfies growth condition and vanishing condition can define a function in $\mathcal{L}_{\mathcal{U}}^{p,\beta}$ or $L_{\mathcal{U}}^{p,\beta}$. To see this, assume that $\vert\psi\Vert_{k}=k$ for $k>0$. Then $\mu*\psi_{l}\in L^{p,\beta}_{0}\cap \mathcal{L}_{\mathcal{U}}^{p,\beta}\cap L^{p,\beta}_{\infty}\cap C^\infty$ for any $l>0$. To do this, let us show that $\mu*\psi_{l}\longrightarrow\mu$ in $\mathcal{L}_{\mathcal{U}}^{p,\beta}$ as $l\longrightarrow 0$. But $L^{p,\beta}\hookrightarrow \mathcal{L}_{\mathcal{U}}^{p,\beta}$. \\
For any $y\in \mathbb{R}^n$, let $\varepsilon>0$ be given and $\xi\in\mathbb{C}^n$, $r>0$ and $l>0$. Then we have
\[\Big(\int_{B(y,r)}(1+\vert \xi\vert^{p})^{\alpha}\vert(\mu*\psi_{l})(\xi)-\mu(\xi)\vert^{p}d\xi\Big)^{\frac{1}{p}}\leq\int_{\mathbb{C}^n}(1+\vert x+z\vert^{k})^{p}\vert\psi_{l}(z)\vert^{p}\Big( \int_{B(y,r)}\vert\mu(\xi-z)-\mu(\xi)\vert^{p}d\xi\Big)^{\frac{1}{p}}dz\]
Since $\mu\in L^{p,\beta}_{0}\cap L^{p,\beta}_{\infty}$, there exist $r_{0},L>0$ such that
\[\frac{1}{r^\beta}\int_{B(y,r)}\vert\mu(\xi-z)-\mu(\xi)\vert^{p}d\xi<\varepsilon\]
for any $r<r_{0}$ or $r>L$. Then we have
\begin{equation}
\label{eqn5.2}
\sup_{\xi\in\mathbb{C}^{n}}\Big(\sup_{r>0}\frac{1}{r^\beta}\int_{B(y,r)}(1+\vert \xi\vert^{p})^{\alpha}\vert(\mu*\psi_{l})(\xi)-\mu(\xi)\vert^{p}d\xi\Big)\leq M\sup\lbrace\varepsilon,C_{r_{0},L(x,l)}\rbrace
\end{equation}
where
$C_{r_{0},R(x,t)}$ has been defined in the proof of Theorem \ref{th*}. For any $\varepsilon>0$ we can find $\delta>0$ such that $\displaystyle \vert\mu(\xi-z)-\mu(\xi)\vert^{p}<\varepsilon^p$ $\forall$ $\xi$ and $z$ with $\vert z\vert<\delta$. We now split the integral into vanishing at origin and at infinity, then
\begin{equation}
\label{eqn5.3}
C_{r_{0},L(x,l)}=\int_{\vert y\vert<\delta}(1+\vert\cdot \vert^{p})^{\alpha}(\cdots)dz+\int_{\vert y\vert\geq \delta}(1+\vert \cdots\vert^{p})^{\alpha}(\cdots)dz
\end{equation}
By uniform continuity of $\mu$, we get
\begin{equation}
\label{eqn5.4}
\int_{\vert y\vert<\delta}(1+\vert\cdot \vert^{p})^{\alpha}(\cdots)dz\leq \varepsilon^{p}r_{0}^{-\beta}\vert B(y,L)\vert(1+\vert\xi \vert^{p})^{\alpha}\int_{\vert y\vert<\delta}\vert\varphi_{l}(z)\vert^{p}dz\leq \vert B(0,1)\vert r_{0}^{-\beta}(1+\vert\xi \vert^{p})^{\alpha} L^{n}\varepsilon^{p}
\end{equation}
for any $y\in\mathbb{R}^n$ and $l>0$. We derive the following inequality using the fact that $\mu\in\mathcal{L}_{\mathcal{U}}^{p,\beta}$ and $\psi\in\mathcal{U}$,
\begin{equation}
\label{eqn5.5}
\int_{\vert y\vert\geq \delta}(1+\vert\cdot \vert^{p})^{\alpha}(\cdots)dz\leq M (1+\vert\xi \vert^{p})^{\alpha}r_{0}^{-\beta} L^{\beta}l\Vert\mu\Vert_{p,\beta}
\end{equation}
Now using (\ref{eqn5.4}) and (\ref{eqn5.5}) in (\ref{eqn5.3}), from (\ref{eqn5.2}) we obtain
\[\sup_{\xi\in\mathbb{C}^{n}}\sup_{x\in\mathbb{R}^{n},r>0}(1+\vert\xi\vert^{p})^{\alpha}r^{-\beta}\int_{B(y,r)}\vert\mu*\psi_{l}-\mu(x)\vert^{p}dx   \lesssim\varepsilon\]
for very small $l>0$. Hence
\[\Big\vert(1+\vert\xi\vert^{p})^{\alpha}\Vert(\mu*\varphi_{l}-\mu\Vert_{p,\beta}\Big\vert<\varepsilon\quad\text{for}\quad l>l_{0}.\]
Hence the result.
\end{proof}

We now present the inclusion of tempered ultradistributions satisfying both vanishing and slow growth conditions since they share some common properties.

\begin{Th}
\label{th5.3}
Given $\beta\in [0,n)$ for $n\in\mathbb{N}$ and $ p\in [1,\infty)$. Let $\mu\in\mathcal{U}'$. Then $f\in\mathcal{L}_{\mathcal{U}}^{p,\beta}$ and $L^{p,\beta}_{0,\infty,*}$ can be represented (described) by $\mu$ in Morrey norm.
\end{Th}
\begin{proof}
We show that function from $\mathcal{L}_{\mathcal{U}}^{p,\beta}$ can be defined by tempered ultradistribution with compact support in Morrey norm. Take $\mu\in\mathcal{L}_{\mathcal{U}}^{p,\beta}$. Let $\delta_{t}:=\delta_{\mathbb{C}^{n}\setminus B(0,t)}$ for $t\in\mathbb{N}$. Then for each $m\in\mathbb{N}$, we get $\delta_{t}=\delta=1$ if $\vert\xi\vert\neq 0$ on the ball $B(0,t)$ and otherwise.\\
Assume that $\varepsilon>0$ is arbitrary. Using the growth condition one finds $r_{0},L>0$ such that
\[\sup_{\xi,y\in\mathbb{C}^n}(1+\vert\xi\vert^{p})^{\alpha}r^{-\beta}\int_{B(y,r)}\vert\mu-\mu_{l}(x)\vert^{p}dx=\sup_{\xi,x\in\mathbb{C}^n}(1+\vert\xi\vert^{p})^{\alpha}r^{-\beta}\int_{B(y,r)}\vert\mu\delta_{l}(x)\vert^{p}dx<\varepsilon\]
Thus for any $r<r_{0}$ or $r>L$, we have
\[\Vert\mu-\mu_{l}\Vert_{p,\beta}^{p}<M\sup\lbrace \varepsilon,C_{0,L}\rbrace\]
By the condition of tempered ultradistributions, we get
\begin{align*}
C_{0,R}(t) &=\sup_{y\in\mathbb{C}^n,r>0}r^{-\beta}\int_{B(y,r)}\vert\mu(x)\vert^{p}\delta_{t}(x)dx\\             &\leq \sup_{y\in\mathbb{C}^{n},r\in[r_{0},L]}r_{0}^{-\beta}\int_{B(y,L)}\vert\mu(x)\vert^{p}\delta_{t}(x)dx     <\varepsilon
\end{align*}
for sufficiently large $t$. Therefore
\[\Vert\mu-\mu_{l}\Vert_{p,\beta}^{p}<\varepsilon\quad\forall\quad l>l_0\]
\end{proof}

From Theorem \ref{th5.3} we present the representation of the spaces $\mathcal{L}_{\mathcal{U}}^{p,\beta}$ and $L_{\mathcal{U}}^{p,\beta}$ in Corollary \ref{cor5.4}.

\begin{Cor}
\label{cor5.4}
Let $\beta\in [0,n)$ and $ p\in [1,\infty)$ be given. The space $\mathcal{U}$ is dense in $\mathcal{L}_{\mathcal{U}}^{p,\beta}$ and $L_{\mathcal{U}}^{p,\beta}$.
\end{Cor}

As we present the result for variation of order placed on the constant exponents, we get
\begin{Th}
\label{th5.5}
Given $p\leq p'\in [1,\infty)$ and $0\leq \beta,u\leq n$. Then the following continuous embedding holds:
\begin{equation}
\label{eqnth5.5}
\mathcal{L}_{\mathcal{U}}^{p',u}\hookrightarrow\mathcal{L}_{\mathcal{U}}^{p,\beta}\quad(\text{respectively},\quad L_{\mathcal{U}}^{p',u}\hookrightarrow L_{\mathcal{U}}^{p,\beta})
\end{equation}
\end{Th}
\begin{proof}
The embedding in (\ref{eqnth5.5} can be achieved through the application of H$\ddot{\text{o}}$lder's inequality under the condition $q(\beta-n)=p(u-n)$. We can also prove it in an alternative method. To do this, we show that $\Vert\cdot\Vert_{\mathcal{L}_{\mathcal{U}}^{p,\beta}}\leq \Vert\cdot\Vert_{\mathcal{L}_{\mathcal{U}}^{p',u}}$.  Let $\mu\in\mathcal{L}_{\mathcal{U}}^{p',u}$ for $0\leq \beta,u\leq n$.\\
Then
\begin{equation}
\label{eqnlocal}
\Vert\mu\Vert_{\mathcal{L}_{\mathcal{U}}^{p',u}}=\sup_{\xi\in\mathbb{C};y\in\Omega;r>0}\Big(\frac{1}{r^u}\int_{\bar{B}(y,r)}(1+\vert\xi\vert^{p})^{\alpha}\vert\mu(\xi)\vert^{p}d\xi\Big)^{\frac{1}{p'}}
\end{equation}
and
\begin{equation}
\label{eqnlocal2}
\Vert\mu\Vert_{\mathcal{L}_{\mathcal{U}}^{p,\beta}}=\sup_{\xi\in\mathbb{C};y\in\Omega;r>0}\Big(\frac{1}{r^\beta}\int_{\bar{B}(y,r)}(1+\vert\xi\vert^{p})^{\alpha}\vert\mu(\xi)\vert^{p}d\xi\Big)^{\frac{1}{p}}
\end{equation}
for $\vert\alpha\vert\leq k$.\\
Since $p\leq p'$, then via application of H$\ddot{\text{o}}$lder's inequality, we have
\begin{equation}
\label{eqnlocal3}
\Big(\frac{1}{r^\beta}\int_{\bar{B}(y,r)}(1+\vert\xi\vert^{p})^{\alpha}\vert\mu(\xi)\vert^{p}d\xi\Big)^{\frac{1}{p}}\leq \Big(\frac{1}{r^u}\int_{\bar{B}(y,r)}(1+\vert\xi\vert^{p})^{\alpha}\vert\mu(\xi)\vert^{p}d\xi\Big)^{\frac{1}{p'}}
\end{equation}
The consideration of (\ref{eqnlocal3}) in (\ref{eqnlocal2}) and (\ref{eqnlocal}) produced the required result.
\end{proof}

Finally we obtain the following embedding result.

\begin{Th}
\label{th5.6}
For $0\leq \beta<n$ and $1< p<\infty$. Then
\[\mathcal{U}\hookrightarrow L_{\mathcal{U}}^{p,\beta}\hookrightarrow\mathcal{L}_{\mathcal{U}}^{p,\beta}\hookrightarrow \mathcal{U}'.\]
\end{Th}
\begin{proof}
The inclusion $L^{p,\beta}_{\mathcal{U}}\hookrightarrow\mathcal{L}_{\mathcal{U}}^{p,\beta}$ is very obvious since the global version is continuously embedded in the local version. Also from definition $\mathcal{U}\hookrightarrow\mathcal{U}'$ since $\mathcal{U}'$ is the dual of the space of all rapidly decreasing ultradifferentiable functions. It remains to show that $\mathcal{L}_{\mathcal{U}}^{p,\beta}\hookrightarrow\mathcal{U}'$. To do this, let $\mu\in \mathcal{L}_{\mathcal{U}}^{p,\beta}$ such that $\displaystyle \sup_{\xi\in\mathbb{C}}\Big\{(1+\vert\xi\vert^{p})^{\alpha}\vert\mu(\xi)\vert^{p}\Big\}<\infty$ for $\vert\alpha\vert\leq k$, $k\in \mathbb{Z}^{+}$. Then we have for any ball $B(y,r),\quad r>0$
\begin{align*}
\Big(\frac{1}{p-1}\Big)\sup_{\xi\in\mathbb{C}}\Big[(1+\vert\xi\vert^{p})^{\alpha}\vert\mu(\xi)\vert^{p-1}\Big]^{\frac{1}{p}} &=\sup_{\xi\in\mathcal{C}^{n}; y\in\Omega}\int_{B(y,r)}(1+\vert\xi\vert^{p})^{\alpha}\vert\mu(\xi)\vert^{p}dx\\
     &\leq \sup_{\substack{\xi\in\mathcal{C}^{n}\\ y\in\Omega;r>0}}r^{-\beta}\int_{B(y,r)}(1+\vert\xi\vert^{p})^{\alpha}\vert\mu(\xi)\vert^{p}dx\\
     &=\Vert\mu\Vert_{\mathcal{L}_{\mathcal{U}}^{p,\beta}}^{p}
\end{align*}
as required.
\end{proof}

\section{Appendix}
\subsection{Some Known Results}
\begin{Th} (Lebesgue's Dominated Convergence Theorem)
\label{thapp1}
Let $(f_{n})$ be a sequence of complex-valued measurable functions on a measure space $\displaystyle \Big(X,\sum,\mu\Big)$. Suppose that the sequence converges to $f$ and is dominated by some integrable function $g$ such that
\[\vert f_{n}(x)\vert\leq \vert g\vert\quad\forall\quad n\in\mathbb{N}\quad\text{and}\quad x\in X\]
Then $f$ is integrable (in the Lebesgue space) and
\[\lim_{n\to\infty}\int_{x}\vert f_{n}-f\vert d\mu=0\]
which also implies that $\displaystyle \lim_{n\to\infty}\int_{X}f_{n}d\mu=\int_{X} fd\mu.$
\end{Th}

\begin{Th} (H$\ddot{\text{o}}$lder Inequality)
\label{thapp2}
Let $\displaystyle \Big(X,\sum,\mu\Big)$ be a measure space and let $\frac{1}{p}+\frac{1}{p'}=1$. Then for all measurable real or complex-valued functions $f$ and $g$ on $X$,
\[\Vert fg\Vert_{1}\leq \Vert f\Vert_{p}\Vert f\Vert_{p'}\]
\end{Th}

\begin{Th} (Young Convolution Inequalitx)
\label{thapp3}
Let $\mu\in L^{p}(\mathbb{R}^{n})$ and $\gamma\in L^{p'}(\mathbb{R}^{n})$ and let $\frac{1}{p}+\frac{1}{p'}=\frac{1}{r}+1$ with $1\leq p,p',r\leq \infty$. Then
\[\Vert \mu*\gamma\Vert_{r}\leq \Vert \mu\Vert_{p}\Vert \gamma\Vert_{p'}\]
Equivalently, if $p,p',r\geq 1$ and $\frac{1}{p}+\frac{1}{p'}+\frac{1}{r}=2$ then
\[\Big\vert\int_{\mathbb{R}^n}\int_{\mathbb{R}^n} \mu(x)\gamma(x-y)h(x)dxdy\Big\vert\leq \Big(\int_{\mathbb{R}^{n}}\vert \mu\vert^{p}\Big)^{\frac{1}{p}}\Big(\int_{\mathbb{R}^{n}}\vert \gamma\vert^{p'}\Big)^{\frac{1}{p'}}\Big(\int_{\mathbb{R}^{n}}\vert h\vert^{r}\Big)^{\frac{1}{r}}\]
\end{Th}

\Ack{The authors would like to gratefullx appreciate those who participated in the preparation of the manuscript for their remarks and comments. This work was sponsored bx the authors involved.}


\begin{thebibliography}{50}

\bibitem{1*} D. R. Adams, J. Xiao. \textit{Morrey spaces in harmonic analysis}. Ark. Mat., 50(2012), 201-230. doi:10.1007/s11512-010-0134-0.


\bibitem{2}J. A. Alonso.:\textit{The Distribution Functions in the Morrey Space}, Proc. Amer. Math. Soc. 83(4), 1981.

\bibitem{3} S. K. Q. Al-Omari, \textit{Cauchy and Poison Integrals of Tempered ultradistributions of Roumieu and Beurling Types}, J. Concr. Appl. Math. 7(1)(2009) 36-47.

\bibitem{4} A. Almeida, S. Samko. \textit{Approximation in Morrey spaces}. J. Funct. Anal. 272(2017), 2392-2411.

\bibitem{almeida} M. F. Almeida, T. Picon, \textit{Fourier transform decay of distributions in Hardy-Morrey spaces}. Results in Mathematics, 79(3) 2024, 1-24.

\bibitem{5} P. K. Banerji and S. K. Q. Al-Omari, \textit{Multipliers and operators on the tempered ultradistribution spaces of Roumieu type for the distributional Hankel-type transformation spaces}, Internat. J. Math. Math. Sci. (2006) 1-7.


\bibitem{8}Y. Brudnyi, \textit{Spaces defined by local approximations}, Tr. Mosk. Mat. Obs. 24(1971)(in Russian); Engl. transl.:Trans, Moscow Math. Soc. 24(1971) 73-139.

\bibitem{9} C.Campanato, \textit{Proprieta di una famigha di spazi funzioni}, Ann. sc. Norm. pisa 18(1964)

\bibitem{11} F. Chiarenza, M. Franciosi, \textit{A generalization of a theorem by C. Miranda}, Ann. Mat. Pura Appl. (4) 161(1992) 285-297.


\bibitem{13} A. Gogatishvili, R. Mutafayev, \textit{New pre-dual space of Morrey space}. J. Math. Anal. Appl. 397(2013) 678-692.


\bibitem{14}D. D. Haroske, S. D. Moura and L. Skrzypczak, \textit{Some embeddings of Morrey spaces with critical smoothness}.

\bibitem{jia} H. Jia, H. Wang, \textit{Decomposition of Hardy-Morrey spaces}. J. Math. Anal. Appl. 354(2009), 99-110

\bibitem{15}T. Kato, \textit{Strong solutions of the Navier-Stokes equation in Morrey spaces}, Bol. Soc. Bras. Mat. 22(1992) 127-155.

\bibitem{18}C. B. Morrey, \textit{On the solutions of quasi-linear elliptic partial differential equations}, Trans. Amer. Math. Soc. 43(1938)

\bibitem{19} E. D. Nursultanov, D. Suragun, \textit{On the convolution operator in Morrey spaces}, J. Funct. Anal. 515(1) 2022.

\bibitem{20}J. Peetre.:\textit{On the theory of $\mathcal{L}^{p,\lambda}$}, J. Funct. Anal. 4(1969), 71-87.

\bibitem{22} H. Rafeiro, N. Samko, S. Samko, \textit{Morrey-Campanato spaces: an overview}. Operator Theory: Advances and Applications. 228(2013), 293-323.


\bibitem{23} C. Roumieu, \textit{Ultradistributions defines sur $(\mathbb{R}*{n})$ etsurcertaines classea de varietese differentiables}, J. $\text{d}'$Analyse Math. 10(1962-63) 153-192.

\bibitem{24} M. Rosenthal, H. Triebel. \textit{Morrey spaces, their duals and preduals}. 28(2015) 1-30.


\bibitem{25}Y. Sawano, H. Gunawan, V. Guliyev, H. Tanaka. \textit{Morrey spaces and related function spaces}. J. Funct. Spaces 2014(2014) 86.


\bibitem{26*}N. Samko, \textit{Maximal, potential and singular operators in vanishing generalized Morrey spaces}, J. Global Optim 57(2013) 1385-1399.

\bibitem{26} J. Sabasti$\tilde{\text{a}}$o e Silva, \textit{Les fonctions analytiques comme ultradistributions dans le calcal op$\breve{\text{e}}$rationnel}, Math. Ann. 136(1958) 58-96.

\bibitem{27} V. S. Vladimorov, \textit{Methods of the theory of generalized functions}, Taylor $\&$ Francis London $\&$, 2002.
\bibitem{28} A. H. Zemanian, \textit{Distribution theory and transform}, Dover Publication Inc. New York, 1987.
\bibitem{29} C. Zorko. \textit{Morrey spaces}, Proc.Amer. Math. Soc. 98 (1986) 586-592.



\end{thebibliography}
\end{document}